\documentclass{amsart}

\usepackage{amsmath}
\usepackage{amsthm}
\usepackage{amsfonts}
\usepackage{amssymb}
\usepackage{graphicx}

\newtheorem{thm}{Theorem}

\newtheorem{lem}{Lemma}
\newtheorem{prop}{Proposition}
\theoremstyle{definition}

\theoremstyle{remark}

\newtheorem{ob}{Observation}
\newcommand{\R}{{\mathbb{R}}}

\newcommand{\Y}{{\text{Y}}}

\begin{document}

\title{Intrinsic 3-linkedness is Not Preserved by $\Y \nabla$ moves}         % Enter your title between curly braces
\author{D. O'Donnol}        % Enter your name between curly braces
\address{Department of Mathematics, Oklahoma State University, Stillwater,  OK 74078, USA}
\email{odonnol@okstate.edu}

%    General info
\subjclass[2010]{Primary 57M15, 57M25;  Secondary 05C10}

\thanks{Supported in part by a NSF-AWM Mathematics Mentoring Travel Grant}

\date{\today}

%\keywords{}
\maketitle

\begin{abstract} This paper introduces a number of new intrinsically 3-linked graphs through five new constructions.  
We then prove that intrinsic 3-linkedness is not preserved by $\Y \nabla$ moves.  
We will see that the graph $M$, which is obtained through a $\Y \nabla$ move on $(PG)^*_*(PG)$, is not intrinsically 3-linked.  
\end{abstract}

\section{Introduction}
A graph, $G$, is \emph{intrinsically knotted} if every embedding of $G$ in $\R^3$ contains a nontrivial knot.  
A link $L$ is \emph{splittable} if there is an embedding of a 2-sphere $F$ in $\R^3\smallsetminus L$ such that  each component of $\R^3\smallsetminus F$ contains at least one
component of $L$.  
If $L$ is not splittable it is called \emph{non-split}.  
A graph, $G$, is \emph{intrinsically linked} if every embedding of $G$ in $\R^3$ contains a non-split link.  
A graph, $G$, is \emph{minor minimal with respect to being intrinsically linked} (or simply \emph{minor minimal intrinsically linked}) if $G$ is intrinsically linked and no minor of $G$ is intrinsically linked.  
The combined work of Conway and Gordon \cite{CG}, Sachs \cite{Sa}, and Robertson, Seymour, and Thomas \cite{RST} fully characterizes intrinsically linked graphs.  
The graphs in the Petersen family, shown in Figure \ref{PF},  are the complete set of minor minimal intrinsically linked graphs.  
So no minor of one of the graphs of the Petersen family is intrinsically linked, and every graph that is intrinsically linked contains one of these graphs as a minor.  
Let the set of the seven graphs of the Petersen family be denoted by $\mathcal{PF}$.  

The concept of a graph being intrinsically linked can be generalized to a graph that intrinsically contains a link of more than two components.  
A graph $G$ is \emph{intrinsically $n$-linked} if every embedding of $G$ in $\R^3$ contains a non-split $n$-component link.  
From here forward we will use \emph{$n$-link} to mean a non-split $n$-component link.  
In this paper we focus on intrinsically 3-linked graphs.  
In Section \ref{back}, we discuss the set of known intrinsically 3-linked graphs and introduce our five new constructions.  
In Section \ref{i3-l}, we prove that each of the new constructions results in an intrinsically 3-linked graph.

%%%%%%%%%%%%%%%%%%%%%%%%%%%%%%%%%%%%%%%%%%%%%%%%%%%%
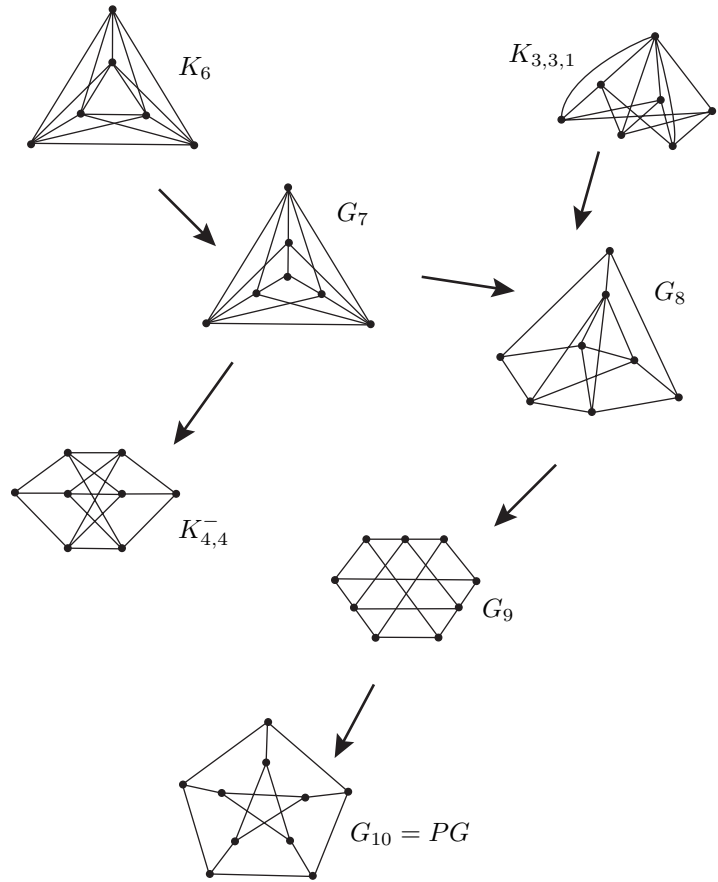
\begin{figure}[h]
\begin{center}
\begin{picture}(260, 330)
\put(0,0){\includegraphics[scale=0.5]{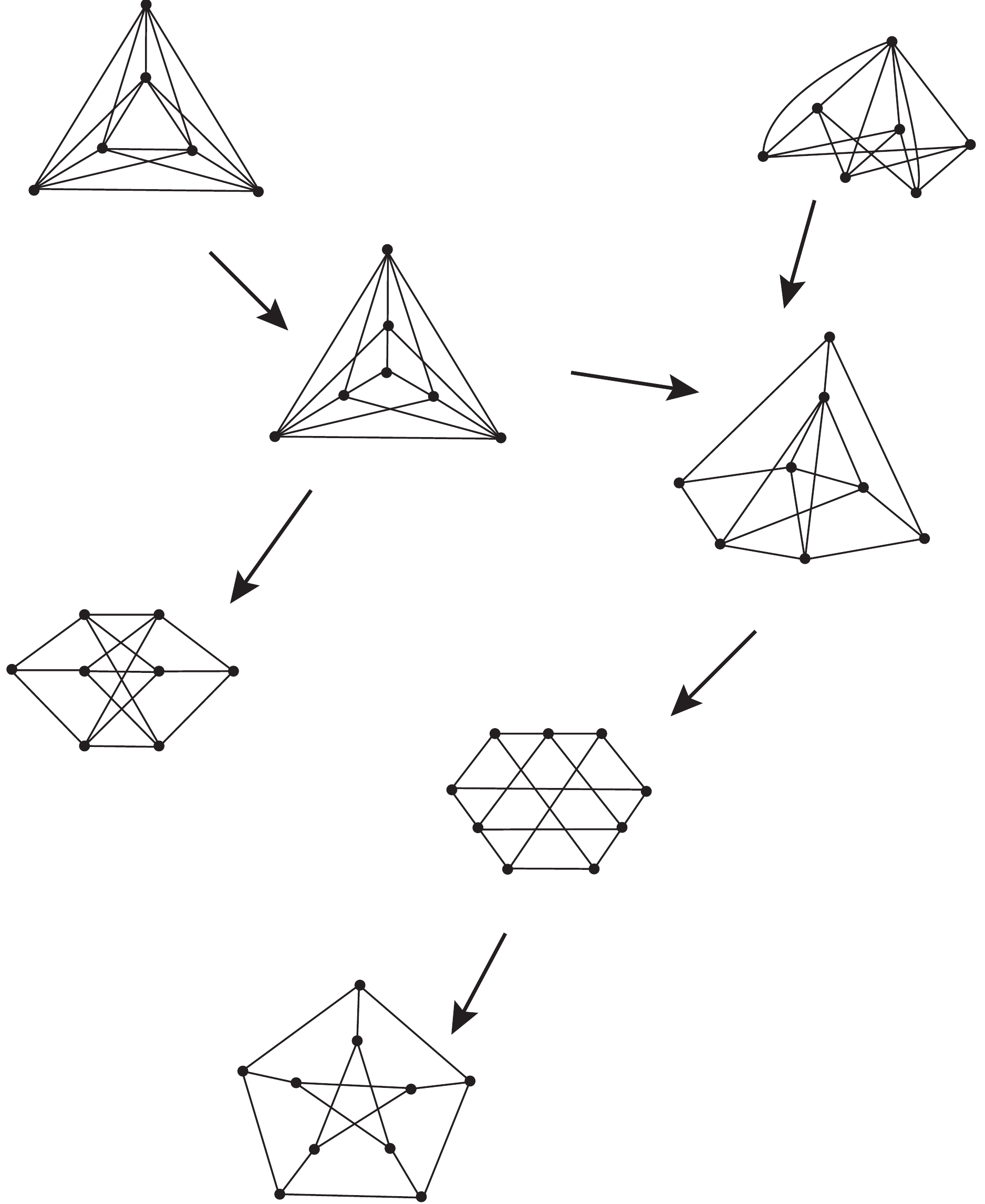}}
\put(65,305){$K_6$}
\put(190,310){$K_{3,3,1}$}
\put(125,250){$G_7$}
\put(245,220){$G_8$}
\put(180,100){$G_9$}
\put(65,130){$K_{4,4}^-$}
\put(130,15){$G_{10}=PG$}
\end{picture}
\caption{This figure shows the graphs in the Petersen family, and the arrows indicate $\nabla\Y$ moves.  }\label{PF}
\end{center}
\end{figure}
%%%%%%%%%%%%%%%%%%%%%%%%%%%%%%%%%%%%%%%%%%%%%%%%%%%%

%%%%%%%%%%%%%%%%%%%%%%%%%%%%%%%%%%%%%%%%%%%%%%%%%%%%
\begin{figure}[htpb!]
\begin{center}
\begin{picture}(200, 70)
\put(0,-10){\includegraphics{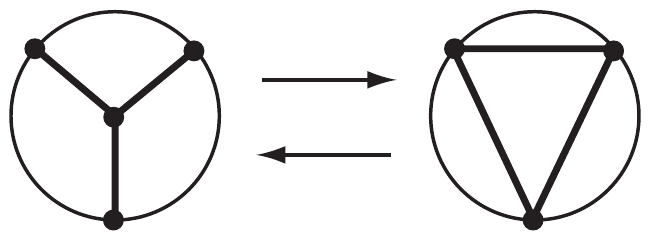}}
\put(72, 41){\small \bf$\Y \nabla$ move}
\put(72, 20){\small \bf $\nabla \Y$ move}
\end{picture}
\caption{The $\Y \nabla$ move and $\nabla\text{Y}$ move.} \label{Ytri}
\end{center}
\end{figure}

%%%%%%%%%%%%%%%%%%%%%%%%%%%%%%%%%%%%%%%%%%%%%%%%%%%%

A \emph{$\text{Y}\nabla$ move} on an abstract graph is where a valance 3 vertex, $v$, together with its adjacent edges are deleted, and three edges are added, one between each pair of vertices that had been adjacent to $v$.  
The reverse move is called a \emph{$\nabla\text{Y}$ move}.  
See Figure \ref{Ytri}.   
In \cite{Sa}, Sachs showed that each graph in the Petersen family, i.e. all those graphs obtained from $K_6$ by $\Y\nabla$ and $\nabla \Y$ moves, is also minor minimal intrinsically linked.  
Motwani, Raghunathan, and Saran \cite{MRS} showed that both intrinsic linkedness and intrinsic knottedness are preserved by $\nabla\Y$ moves.  
Their proof that intrinsic linkedness is preserved by $\nabla\Y$ moves immediately generalizes to show that intrinsic $n$-linkedness is also preserved by $\nabla\Y$ moves.  
Robertson, Seymour, and Thomas \cite{RST} showed that $\Y\nabla$ moves also preserve intrinsic linkedness.  
On the other hand, Flapan and Naimi \cite{FN} showed that $\Y\nabla$ moves do not preserve intrinsic knottedness.
It is not known if intrinsic $n$-linkedness is preserved by $\Y\nabla$ moves in general.   
The work in \cite{RST} showed that intrinsic $2$-linkedness is preserved by $\Y\nabla$ moves.  
While the family of minor minimal intrinsically linked graphs (also minor minimal intrinsically 2-linked graphs) is connected by $\Y\nabla$ and $\nabla \Y$ moves, the family of minor minimal intrinsically 3-linked graphs is not \cite{FFNP}.  
It is also known that if the graph resulting from a $\Y\nabla$ move on a minor minimal intrinsically $n$-linked graph is intrinsically $n$-linked, then it is minor minimal intrinsically $n$-linked \cite{BDLST}.  
In Section \ref{Ytri}, we prove that intrinsic 3-linkedness is not preserved by $\Y\nabla$ moves.  

\subsection*{Acknowledgements}
The author would like to thank Dorothy Buck, Erica Flapan, Kouki Taniyama and R. Sean Bowman for helpful conversations and their continued support.

%%%%%%%%%%%%%%%%%%%%%%%%%%%%%%%%%%%%%%
%%%%%%%%%%%%%%%%%%%%%%%%%%%%%%%%%%%%%%
\section{Intrinsically 3-linked graphs}\label{back}
There are a number of graphs already known to be intrinsically 3-linked.  
Figure \ref{I3L} shows all those graphs that have been shown to be intrinsically 3-linked, where no minor is known to be intrinsically 3-linked (only $G(2)$ is known to be minor minimal intrinsically 3-linked).  
In \cite{FNP}, Flapan, Naimi, and Pommersheim investigate intrinsically 3-linked graphs (or \emph{intrinsically triple linked} graphs).  
They proved that the complete graph on ten vertices, $K_{10}$ is the smallest complete graph to be intrinsically 3-linked. 
Bowlin and Foisy \cite{BF} also looked at intrinsically 3-linked graphs. 
They exhibited two different subgraphs of $K_{10}$ that are also intrinsically 3-linked, the graph resulting from removing two disjoint edges from $K_{10}$, call it $K_{10}-\{2\text{ edges}\}$, and the graph obtained by removing four edges incident to a common vertex from $K_{10}$, call it $K_{10}^*$.  
So $K_{10}$ is not minor minimal intrinsically 3-linked. 
They also described two constructions that give intrinsically 3-linked graphs, the first being the graph that results from identifying an edge of $G_1$ with an edge of $G_2$, when $G_1$ and $G_2$ are either $K_7$ or $K_{4,4}$; we will call this graph $G_1|G_2$.  
Since all of the edges of $K_7$ are equivalent as are those of $K_{4,4}$ this gives rise to three graphs.  
One of these graphs, $K_{4,4}|K_{4,4}$ (also called $J$), was previously shown to be intrinsically 3-linked in \cite{FFNP}.  
The second is the graph obtained by connecting two graphs $G_1,G_2\in \mathcal{PF}$ by a 6-cycle where the vertices of the 6-cycle alternate between $G_1$ and $G_2$.  
We will call such a graph $(G_1CG_2)_i$.  
The subscript $i$ is given because there can be multiple ways to combine the same two graphs in this way.  
Not all of the vertices of each of the graphs of the Petersen family are equivalent, so there are many different ways to form the 6-cycle.   
Thus, there will be many more than $7+{7\choose 2}=28$ such graphs that one might first expect from combining the seven different graphs of $\mathcal{PF}$ in this construction.    
Due to the large number of graphs of this form they are not drawn in Figure \ref{I3L} but instead a pictorial representation of any such graph is shown.  
Flapan, Foisy, Naimi, and Pommersheim addressed the question of minor minimal intrinsically $n$-linked graphs in \cite{FFNP}, where they constructed a family of minor minimal intrinsically $n$-linked graphs. 
The minor minimal intrinsically $3$-linked graph they constructed was called $G(2)$, shown in Figure \ref{I3L}.

%%%%%%%%%%%%%%%%%%%%%%%%%%%%%%%%%%%%%%%%%%%%%%%%%%%%
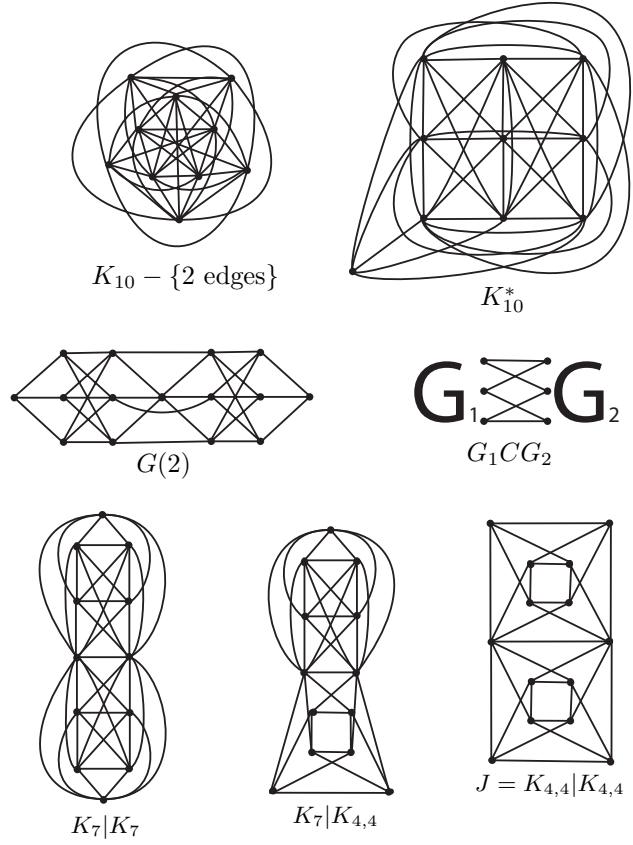
\begin{figure}[h]
\begin{center}
\begin{picture}(260, 330)
\put(0,0){\includegraphics[scale=0.7]{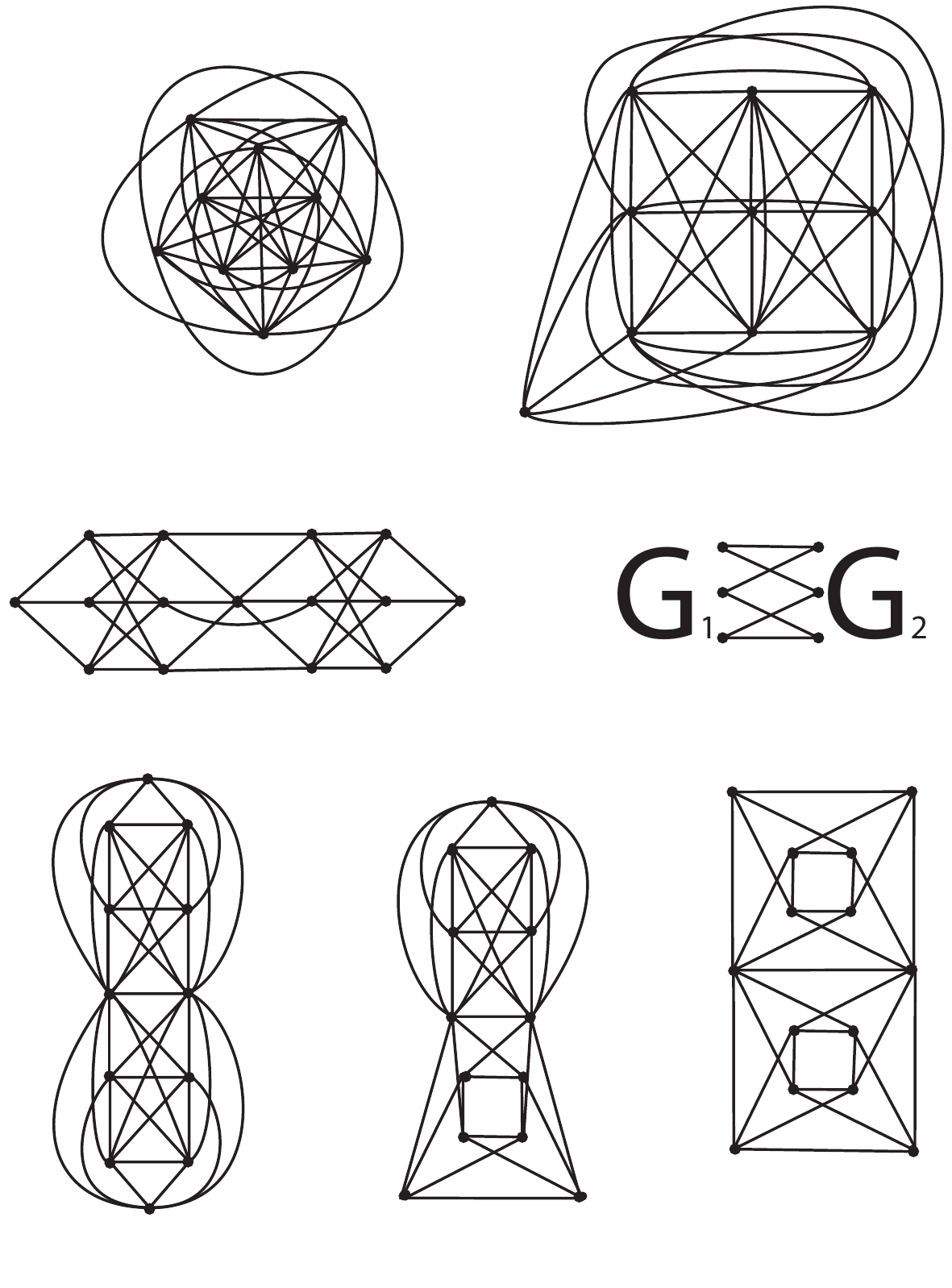}}
\put(33,215){$K_{10}-\{2\text{ edges}\}$}
\put(180,207){$K_{10}^*$}
\put(50,144){$G(2)$}
\put(175,148){$G_1CG_2$}
\put(25,7){\small\bf$K_7|K_7$}
\put(109,11){\small\bf$K_7|K_{4,4}$}
\put(178,23){\small\bf$J=K_{4,4}|K_{4,4}$}
\end{picture}
\caption{The set of of previously known intrinsically 3-linked graphs, for which no known minor is intrinsically 3-linked. (The graph $G(2)$ is known to be minor minimal.)}\label{I3L}
\end{center}
\end{figure}
%%%%%%%%%%%%%%%%%%%%%%%%%%%%%%%%%%%%%%%%%%%%%%%%%%%%

In Section \ref{i3-l}, we prove that the following five constructions give rise to intrinsically 3-linked graphs: 
Let $G_1, G_2 \in \mathcal{PF}$ and let $v_i$ be a vertex in the graph $G_i$.  
Let the set of adjacent vertices to the vertex $v_i$ be $A_i$, for $i=1, 2$.

\smallskip
{\bf Construction 1}:   Let the graph obtained by adding the edges between $v_1$ and all but one of the vertices of $A_2$ and  the edges between $v_2$ and all but one of the vertices of $A_1$ to the graphs $G_1$ and $G_2$ be called $(G_1,v_1)^*_*(G_2, v_2)$.  
We call this the \emph{double star construction}.  

\smallskip
{\bf Construction 2}:  Let the graph obtained by identifying the vertices $v_1$ and $v_2$,  adding a vertex $x$ and the edges from $x$ to all but one of the vertices in the set $A_1$ and all but one of the vertices in the set $A_2$ be called $(G_1,v_1)^*_x(G_2, v_2)$.  

\smallskip
{\bf Construction 3}:  The graph $K_{4,4}^-$ has two vertices of  valence three, label them $x$ and $y$.  
Let $G$ be one of the graphs of the Petersen family, let $v$ be one of the vertices of $G$, and let $A$ be the set of vertices adjacent to $v$ in $G$.  
Let the graph obtained by identifying the two vertices $x$ and $v$, and adding edges between $y$ and all but one of the vertices in $A$ be called $K_{4,4}(G,v)$.  

\smallskip
{\bf Construction 4}: Let the \emph{vertex identification construction} be the construction where a graph $H$ is formed by adding edges between the sets of vertices $A_1$ and $A_2$, such that between every pair of vertices of $A_1$ and pair of vertices of $A_2$ there is at least one edge joining a vertex from $A_1$ to a vertex from $A_2$ and then identifying the vertices $v_1$ and $v_2$ to get a single vertex $x$.  Let the set of added edges between $A_1$ and $A_2$ be $E_{n,m}$, where $|A_1|=n$ and $|A_2|=m$. 

\smallskip
{\bf Construction 5}: Let $V_i$ be the full vertex set of $G_i$.  Let $(G_1)^\equiv_=(G_2)$ be the graph obtained by adding five disjoint edges between $V_1$ and $V_2$ to the graphs $G_1$ and $G_2$.

The vertex is dropped from the notation for Constructions 1, 2 and 3, if the vertices of the graph $G_1$ or $G_2$ are equivalent.  
\smallskip
These constructions introduce numerous new intrinsically 3-linked graphs.  
For example, the graphs $(K_6)^*_*(K_6)$, $(K_6)^*_x(K_6)$ are both subgraphs of $K_7|K_7$.  
See Figure \ref{mm}.  
So these two graphs together with all of the graphs that can be obtained from them by $\nabla\Y$ moves were previously unknown to be intrinsically 3-linked.  
Similarly, $K_{4,4}(K_6)$  is a subgraph of $K_7|K_{4,4}$.  
So the graph $K_{4,4}(K_6)$ introduces a set of new intrinsically 3-linked graphs.  \\

%%%%%%%%%%%%%%%%%%%%%%%%%%%%%%%%%%%%%%
%%%%%%%%%%%%%%%%%%%%%%%%%%%%%%%%%%%%%
\section{New intrinsically 3-linked graphs}\label{i3-l}

In this section we prove that the five new constructions explained in Section \ref{back} give intrinsically 3-linked graphs.  
These constructions exploit some nice properties of the graphs in the Petersen family.  
\begin{ob} For each $G\in\mathcal{PF}$ every pair of disjoint cycles contains all of the vertices of $G$.  
\end{ob}
Thus every embedding of a graph from the Petersen family in $\R^3$ not only contains a 2-link but contains a 2-link which contains all of the vertices of the graph.  
It is known that, for any $G\in\mathcal{PF}$, every embedding of $G$ in $\R^3$ contains a two component link with odd linking number \cite{CG, Sa}.  
So we will work with linking mod (2) and denote the mod (2) linking number of two simple closed curves $L$ and $J$ by $\omega (L,J)$.

%%%%%%%%%%%%%%%%%%%%%%%%%%%%%%%%%%%%%%%%%%%%%%%%%%%%
\begin{figure}[h]
\begin{center}
\includegraphics[scale=0.4]{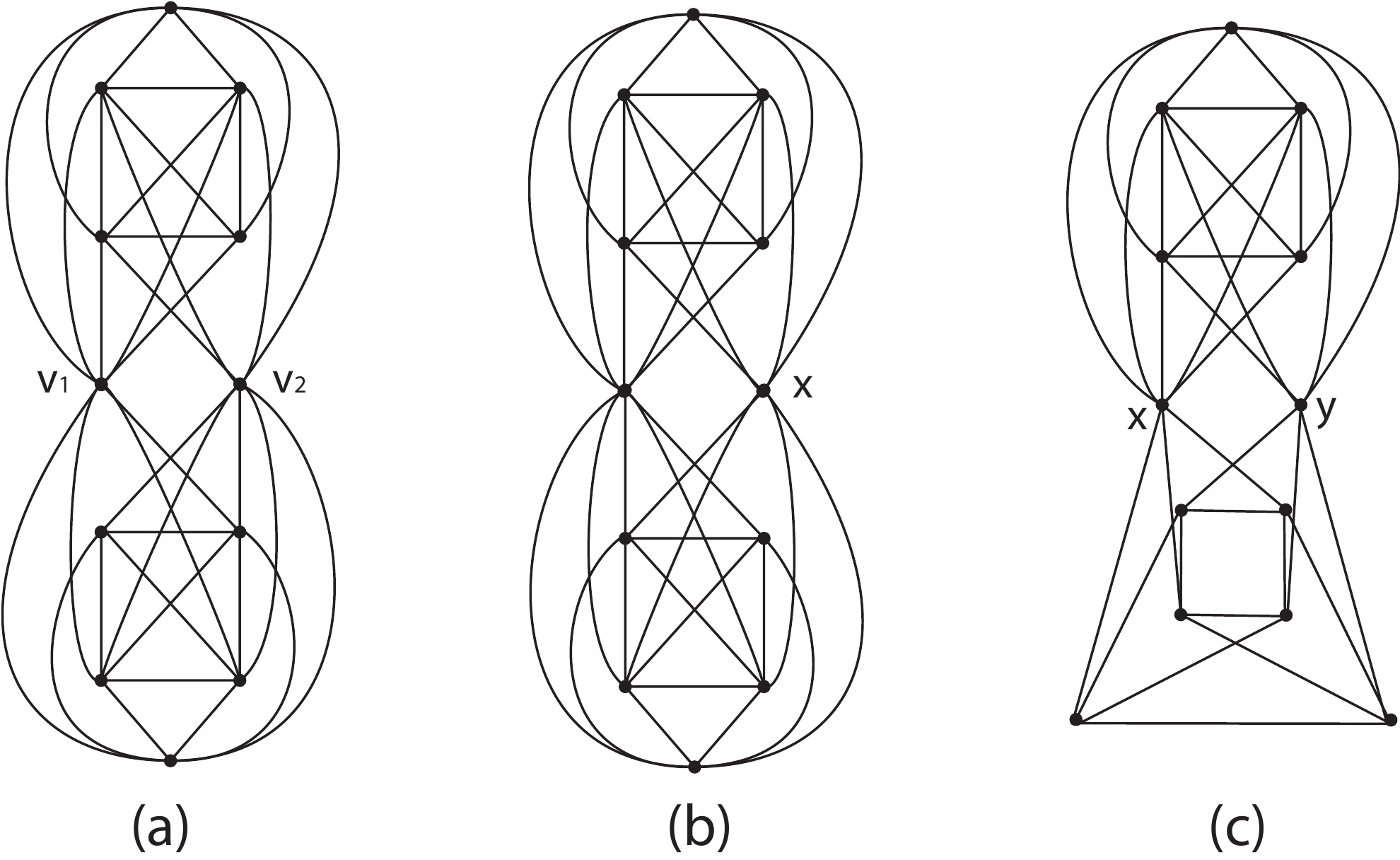}
\caption{The graphs (a) $(K_6)^*_*(K_6)$, (b) $(K_6)^*_x(K_6)$, and
 (c) $K_{4,4}(K_6)$.} \label{mm}
\end{center}
\end{figure}
%%%%%%%%%%%%%%%%%%%%%%%%%%%%%%%%%%%%%%%%%%%%%%%%%%%%
\smallskip

We will use the following lemma, proved in \cite{BF}, to prove that Construction 1 of the previous section gives rise to an intrinsically 3-linked graph.  

\begin{lem}\label{2path}
In an embedded graph with mutually disjoint simple closed
curves, $C_1$, $C_2$, $C_3,$ and $C_4$, and two disjoint paths
$x_1$ and $x_2$, such that $x_1$ and $x_2$ begin in $C_2$ and end
in $C_3$, if $\omega(C_1,C_2)=\omega(C_3,C_4)=1$ then the embedded
graph contains a non-splittable 3-component link.
\end{lem}

\begin{prop}\label{starstar}  Let $(G_1,v_1)^*_*(G_2, v_2)$ be a graph obtained via Construction 1.  Then $(G_1,v_1)^*_*(G_2, v_2)$ is intrinsically 3-linked. 
\end{prop}

\begin{proof}  
Fix an arbitrary embedding of $(G_1,v_1)^*_*(G_2, v_2)$.  
Since $G_1$ is a graph in the Petersen family we know it must contain a 2-link and that $v_1$ must be in one of the components of the 2-link.  
Let the component that contains $v_1$ be called $C_2$ and the other component be called $C_1$.  
Note that $\omega(C_1,C_2)=1$.  
Similarly, $G_2$ must contain a 2-link and that $v_2$ must be in one of the components of the 2-link.  
Let the component that contains $v_2$ be called $C_3$ and the other component be called $C_4$, note that $\omega(C_3,C_4)=1$.  
Since $v_1$ is in $C_2$ two of the vertices adjacent to $v_1$ are also in $C_2$.  
At least one of these vertices must be adjacent to $v_2$, call it $a$.   
Note, the edge $\overline{v_1 a}$ goes between $C_2$ and $C_3$.  
Next, since $v_2$ is in $C_3$ two of the vertices adjacent to $v_2$ are also in $C_3$ and at least one of them is adjacent to $v_1$, call it $b$.  
The edge $\overline{b v_2}$ also goes between $C_2$ and $C_3$.   
Notice that $\overline{v_1 a}$ and $\overline{b v_2}$ cannot be the same edges by construction.  
Thus by Lemma \ref{2path} we see that the chosen embedding contains a 3-link, and so $(G_1,v_1)^*_*(G_2, v_2)$ is intrinsically 3-linked.  
\end{proof}

%Construction 1 introduces a number of new intrinsically 3-linked graphs among them is $(K_6)^*_*(K_6)$.  
%The graph  $(K_6)^*_*(K_6)$ is a subgraph of $K_7|K_7$.   
%So among others $(K_6)^*_*(K_6)$ together will all graphs obtained from it by $\nabla Y$ moves are new intrinsically 3-linked graph.  
%\smallskip

To prove Proposition \ref{con2}, we will use the following lemma which appears in \cite{FFNP}:

\begin{lem}\label{3Lpath}
Suppose that $G$ is a graph embedded in $\R^3$ and contains the simple closed curves $C_1, C_2, C_3,$ and $C_4$. 
Suppose that $C_1$ and $C_4$ are disjoint from each other and both are disjoint from $C_2$ and $C_3$, and that $C_2$ and $C_3$ intersect in precisely one vertex $x$. 
Also, suppose there are vertices $u\neq x$ in $C_2$ and $v\neq x$ in $C_3$ and a path P in $G$ with endpoints $u$ and $v$ whose interior is disjoint from each $C_i$.
 If $\omega(C_1,C_2)=\omega(C_3,C_4)=1$, then there is a non-splittable 3-component link in $G$.
\end{lem}

\begin{prop}\label{con2}  Let $(G_1,v_1)^*_x(G_2, v_2)$ be a graph obtained via Construction 2, then $(G_1,v_1)^*_x(G_2, v_2)$ is intrinsically 3-linked. 
\end{prop}

\begin{proof} Let $A_i$ be the sets vertices and $x$ be the vertex as described in Construction 2, in the previous section.   Fix an arbitrary embedding of $(G_1,v_1)^*_x(G_2, v_2)$.  Since $G_1$ is in the Petersen family we know it must contain a 2-link and that $v_1$ must be in one of the 
components of the 2-link.  Let the component that contains $v_1$ be called $C_2$ and the other component be called $C_1$; note that $\omega(C_1,C_2)=1$.  
Similarly, $G_2$ must contain a 2-link and that $v_2$ must be in one of the components of the 2-link.  
Let the component that contains $v_2$ be called $C_3$ and the other component be called $C_4$; note that $\omega(C_3,C_4)=1$.  
Since $v_1$ is in $C_2$ there are two vertices in $A_1$ that are also in $C_2$ and at least one of them is adjacent to $x$.  Call it $a_i$.  
Similarly, since $v_2$ is in $C_3$ there are two vertices of $A_2$ that are also in $C_3$ and at least one of them is adjacent to $x$.  Call it $b_i$.    
Let the path consisting of the two edges $\overline{a_ix}$ and $\overline{xb_i}$ be called $P$.  The path $P$ goes from $C_2$ to $C_3$.  
So by Lemma \ref{3Lpath} we see that the embedding contains a 3-link.  Thus $(G_1,v_1)^*_x(G_2, v_2)$ is intrinsically 3-linked.  
\end{proof}

The graph $K_{4,4}(K_6)$ is shown in Figure \ref{mm}.  We will use the following lemmas in the proof of the next proposition about Construction 3.
See Section \ref{back} for the constructions.    

\begin{lem}\label{k44}
\cite{Sa} Let $K_{4,4}$ be embedded in $\R^3$, then every edge of
$K_{4,4}$ is in a component of a 2-link.
\end{lem}

\begin{lem}\label{3l}
\cite{FNP} Suppose that $G$ is a graph embedded in $\R^3$ that
contains the simple closed curves $C_1, C_2, C_3,$ and $C_4$.
Suppose that $C_1$ and $C_4$ are disjoint from each other and both
are disjoint from $C_2$ and $C_3$, and that $C_2\cap C_3$ is an
arc.  If $\omega(C_1,C_2)=1$ and $\omega(C_3,C_4)=1$, then there
is a non-split 3-component link in $G$.
\end{lem}

\begin{prop} Let $K_{4,4}(G,v)$ be a graph obtained via Construction 3, then $K_{4,4}(G,v)$ is intrinsically 3-linked. 
\end{prop}

\begin{proof}Let $x$, $y$, and $A$ be as described in Construction 3.  Fix an arbitrary embedding of $K_{4,4}(G,v)$.  
Since $G$ is one of the graphs from the Petersen family it contains a 2-link $C_1\cup C_2$ which contains the vertex $v=x$, without loss of generality let $x$ be in $C_2$.  
Since $x$ is in $C_2$ two of the vertices adjacent to $x$, are also in $C_2$.  
So at least one of these vertices in $C_2$ is adjacent to $y$, call the vertex $a$.  
Label the path P, that is comprised of the two edges $\overline{xa}$ and $\overline{ay}$.  Notice $\overline{xa}\in C_2$ and $\overline{ay}\notin C_1\cup C_2$.  
Now the subgraph $K_{4,4}^-$ together with the path $P$ form a subdivision of $K_{4,4}$, i.e.~this can be viewed as $K_{4,4}$ where $P$ is one if the edges.  
By Lemma \ref{k44} for every embedding of $K_{4,4}$ each edge is contained in a 2-link, so $P$ is contained in a 2-link $C_3\cup C_4$.  
Without loss of generality, let $P$ be an edge of $C_3$.  
So $C_2\cap C_3=\overline{xa}$ is an arc,  $\omega(C_1,C_2)=1$ and $\omega(C_3,C_4)=1$.  Thus by Lemma \ref{3l} the embedding of  $K_{4,4}(G,v)$ contains a 3-link.  
\end{proof}

%%%%%%%%%%%%%%%%%%%%%%%%%%%%%%%%%%%%%%%%%%%%%%%%%%%%
\begin{figure}[h]
\begin{center}
\includegraphics[scale=0.5]{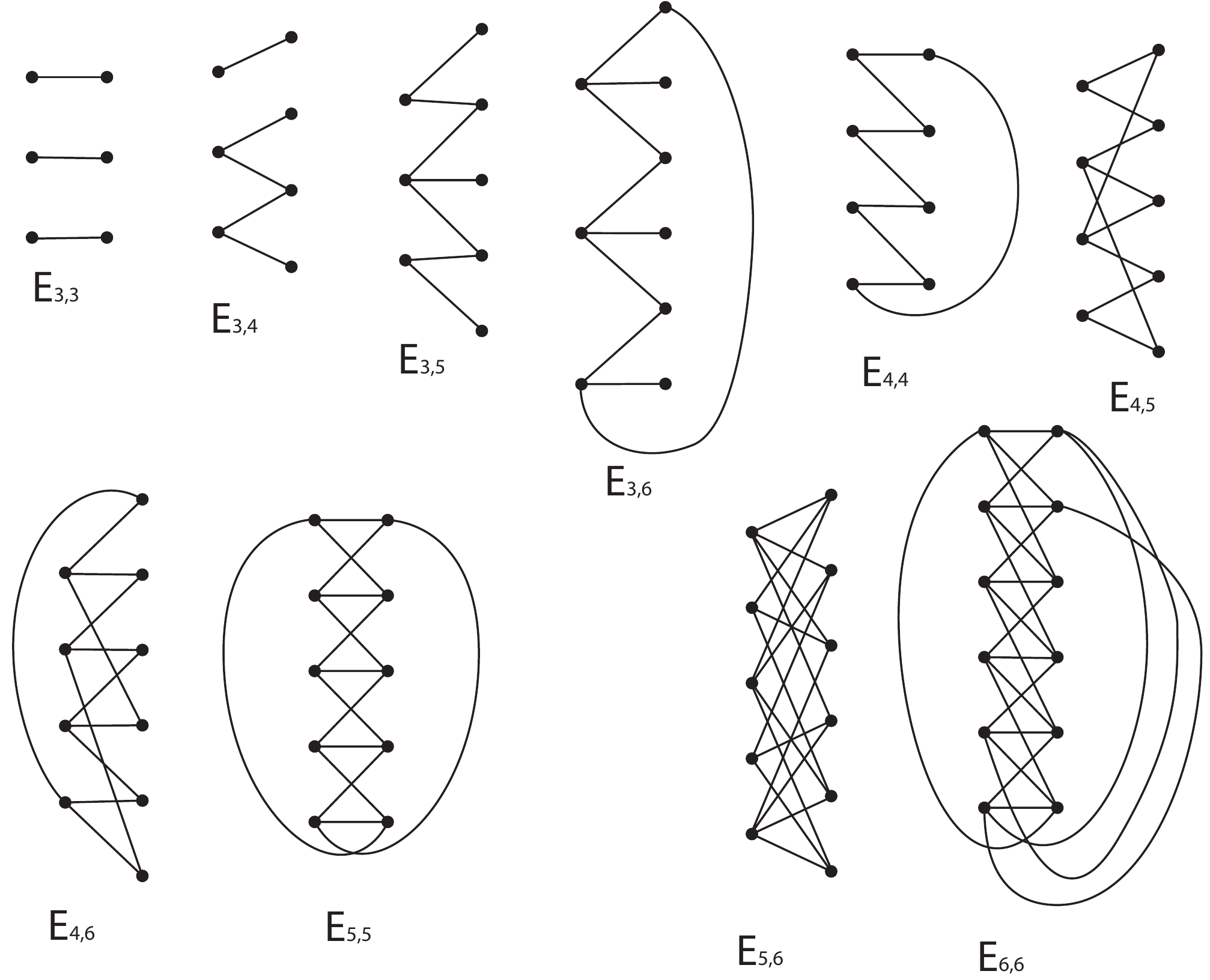}
\caption{Let $|A_1|=n$ and $|A_2|=m$.  This figure shows possible sets of added edges $E_{n,m}$ for
the vertex identification constructions for all different possible sizes of the vertex set $A_1$ and $A_2$.  
For each of them the vertex sets $A_1$ and $A_2$ are shown vertically and additional edges $E_{n,m}$ are shown.  } \label{addedge}
\end{center}
\end{figure}
%%%%%%%%%%%%%%%%%%%%%%%%%%%%%%%%%%%%%%%%%%%%%%%%%%%%

%\subsection*{Construction 4} Let $G_1$ and $G_2$ be graphs from the Petersen family.  Let $x_i$
%be a vertex of $G_i$ and let $A_i$ be the set of vertices adjacent
%to $x_i$, for $i\in\{1,2\}$. Let the \emph{vertex identification
%construction} be the construction where a graph $H$ is formed by
%adding edges between the sets of vertices $A_1$ and $A_2$, such
%that between every pair of vertices of $A_1$ and pair of vertices
%of $A_2$ there is at least one edge joining a vertex from $A_1$ to
%a vertex from $A_2$ and then identifying the vertices $x_1$ and
%$x_2$ to get a single vertex $x$.  Let the set of added edges
%between $A_1$ and $A_2$ be $E_{n,m}$ where $|A_1|=n$ and $|A_2|=m$. 

%The graph $G(2)$ can be thought of as the graph formed by doing the vertex
%identification construction with the valence 3 vertices of two
%$K_{4,4}^-$ graphs.  Recall, in \cite{FFNP},  they prove that $G(2)$ is minor minimal
%intrinsically 3-linked.  

Recall Construction 4, the vertex identification construction, where a set of edges $E_{n,m}$ is added between the two sets of vertices $A_1$ and $A_2$.  
For the full definition refer to Section \ref{back}.  
The graphs in the Petersen family have vertices of valence 3, 4, 5, and 6.  Let  $|A_1|=n$ and $|A_2|=m$.  We want to construct $E_{n,m}$, a set of edges between $A_1$ and $A_2$ such that given any pair of vertices from $A_1$ and any pair of vertices from $A_2$ there is an edge between two of the vertices from the chosen pair.  So each pair of vertices from $A_1$ must be connected to $m-1$ vertices from $A_2$.  To reduce the total number of edges needed we divide the edges evenly between the two vertices, so each vertex of $A_1$ is connected to $\frac{m-1}{2}$ vertices of $A_2$.   Suppose $m\geq n$, this gives a lower bound of $|E_{n,m}|=(\frac{m-1}{2})n$, if $m$ is odd, and $|E_{n,m}|=(\frac{m}{2})(n-1)+\frac{m-2}{2}$ if $m$ is even.  Figure \ref{addedge} shows possible sets of edges $E_{n,m}$ to be added
between $A_1$ and $A_2$ for all possible combinations of valence.  It can be checked that there sets of vertices satisfy the criterion.  However these are not the only possible $E_{n,m}$ sets, and it is not known if they are optimal.  In the case of $n=m=3$ then $|E_{3,3}|=3$ is the lower bound but in all other examples given $|E_{n,m}|$ is greater than the lower bound obtained.

\begin{prop}\label{VIdI3L}Any graph $H$ constructed through the vertex identification
construction of graphs $G_1$ and $G_2$ in the Petersen family is
intrinsically 3-linked.
\end{prop}

\begin{proof}
Consider an arbitrary embedding of $H$.  Let the notation be as
in construction 4: the identified vertex is labelled $x$, the sets of vertices
adjacent to $x$ in the subgraphs $G_1$ and $G_2$, respectively,
are labelled $A_1$ and $A_2$, and the set of edges between $A_1$
and $A_2$ is $E_{n,m}$.  
Since every vertex of a Petersen graph is
contained in every link in the embedding, $x$ is in one of the
components of the link in each $G_i$. Let the components of the
link in $G_1$ be labelled $C_1$ and $C_2$, with the vertex $x$ in
the component $C_2$, and let the components of the link in $G_2$
be labelled $C_3$ and $C_4$, with the vertex $x$ in the component
$C_3$.  A pair of vertices from the set $V_1$ is also part of
$C_2$ and, similarly, a pair of vertices from the set $V_2$ is
also part of $C_3$.  By construction, there is an edge, $e$, of
the set $E_{n,m}$ between $C_2$ and $C_3$.  Since none of the edges of
$E_{n,m}$ are contained in $G_1$ or $G_2$, the interior of the edge
$e$ is disjoint from the links $C_1\cup C_2$ and $C_3\cup C_4$.
Thus, by Lemma \ref{3Lpath}, H contains a 3-link.
\end{proof}

\begin{prop}The graph $(G_1)^\equiv_=(G_2)$ obtained by Construction 5, is intrinsically 3-linked.  
\end{prop} 
 
 \begin{proof}  Let $V_i$ be the vertex set of $G_i$, let the set of five added edges be $E$.  
Fix an embedding of $(G_1)^\equiv_=(G_2)$.  
Since $G_1,G_2\in \mathcal{PF}$, $G_1$ contains a 2-link that contains all of the vertices of $V_1$, call the link $C_1\cup C_2$, and  $G_2$ contains a 2-link that contains all of the vertices of $V_2$, call the link $C_3\cup C_4$. 
 Because all of the vertices are in one of the 2-links each edge of $E$ will go between components of the different 2-links.  
There are four different pairs of components that can be connected by the said edges, so by the pigeonhole principle two of the edges must go between the same pair of components.  
 By Lemma \ref{2path} the embedding contains a 3-link.  
 Thus $(G_1)^\equiv_=(G_2)$ is intrinsically 3-linked.  
 \end{proof}

%%%%%%%%%%%%%%%%%%%%%%%%%%%%%%%%%%%%%%
%%%%%%%%%%%%%%%%%%%%%%%%%%%%%%%%%%%%%
\section{Intrinsic 3-linkedness is not preserved by $\Y \nabla$ moves}\label{Ytri}

In this section, we show that intrinsic 3-linkedness is not preserved by $\Y \nabla$ moves.  

%%%%%%%%%%%%%%%%%%%%%%%%%%%%%%%%%%%%%%%%%%%%%%%%%%%%
\begin{figure}[h]
\begin{center}
\begin{picture}(260, 220)
\put(10,0){\includegraphics[scale=0.8]{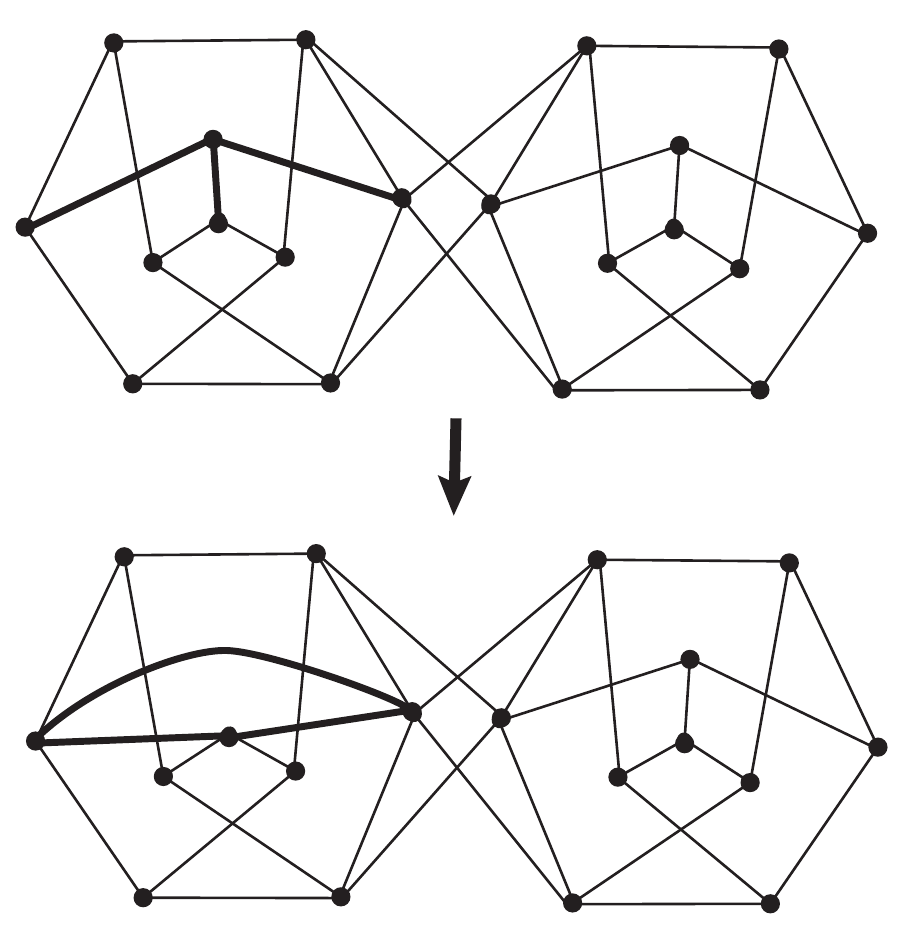}}
\put(200,140){$(PG)^*_*(PG)$}
\put(200,15){\bf$M$}
\end{picture}
\caption{The graph $M$ is obtained from $(PG)^*_*(PG)$ by a $\Y\nabla$ move on the indicated bold edges.  } \label{Ytri}
\end{center}
\end{figure}
%%%%%%%%%%%%%%%%%%%%%%%%%%%%%%%%%%%%%%%%%%%%%%%%%%%%

%%%%%%%%%%%%%%%%%%%%%%%%%%%%%%%%%%%%%%%%%%%%%%%%%%%%
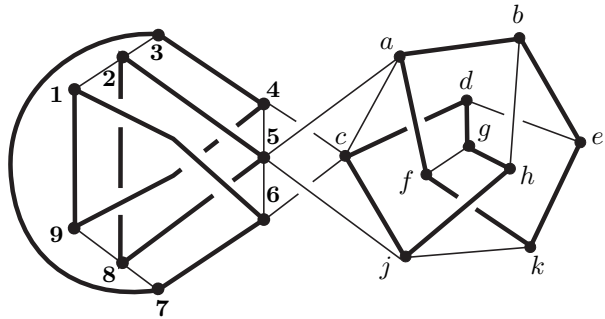
\begin{figure}[h]
\begin{center}
\begin{picture}(260, 130)
\put(0,10){\includegraphics[scale=0.8]{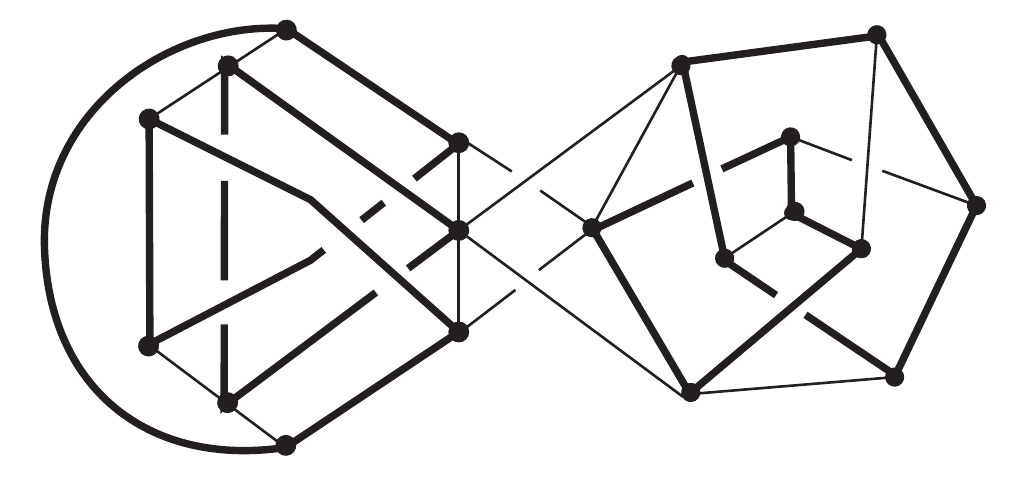}}
\put(25,88){\small\bf 1}
\put(45,94){\small\bf 2}
\put(63,104){\small\bf 3}
\put(107,90){\small\bf 4}
\put(107,72){\small\bf 5}
\put(107,50){\small\bf 6}
\put(65,7){\small\bf 7}
\put(45,20){\small\bf 8}
\put(25,35){\small\bf 9}
\put(150,108){\bf$a$}
\put(200,118){\bf$b$}
\put(133,72){\bf$c$}
\put(180,93){\bf$d$}
\put(230,72){\bf$e$}
\put(157,55){\bf$f$}
\put(187,75){\bf$g$}
\put(203,57){\bf$h$}
\put(150,23){\bf$j$}
\put(207,23){\bf$k$}
\end{picture}
\caption{The spatial graph $f(M)$.  An embedding of $M$ that does not contain a 3-link.  } \label{f(M)}
\end{center}
\end{figure}
%%%%%%%%%%%%%%%%%%%%%%%%%%%%%%%%%%%%%%%%%%%%%%%%%%%%
 
 \begin{thm} Intrinsic 3-linkedness is not preserved by $\Y \nabla$ moves.
 \end{thm}

\begin{proof} 
We begin with $(PG)^*_*(PG)$, since all of the vertices of $PG$ are equivalent there is a single graph that can be obtained throughout the double star construction with two Petersen graphs.   
By Proposition \ref{starstar}, $(PG)^*_*(PG)$ is intrinsically 3-linked.  
Let the graph obtained by a $\Y\nabla$ move on $(PG)^*_*(PG)$ as indicated in Figure \ref{Ytri} be called $M$.  

We claim that, the graph $M$ is not intrinsically 3-linked.
Consider the embedding $f(M)$ shown in Figure \ref{f(M)}.  
Let the vertices be labelled as indicated.  
Let $K_1$ be the embedded subgraph defined by the vertices 1, 2, 3, 4, 5, 6, 7, 8, 9 and the edges between them in $f(M)$, and let $K_2$ be the embedded subgraph defined by the vertices $a, b, c, d, e, f, g, h, j, k$ and the edges between them.  
For $f(M)$ to contain a 3-link, the 3-link must be in both the embedded subgraphs $K_1$ and $K_2$, since neither contains three disjoint simple closed curves on their own.  
Since $K_1$ and $K_2$ are disjoint and there is no linking between them, two of the edges joining the subgraphs must also be in the 3-link.  
The subgraph $K_1$ contains a single linked pair of cycles, indicated with thickened edges.  
Similarly, in $K_2$ there is a single linked pair of cycles, indicated with thickened edges.  
All other cycles in the subgraphs $K_1$ and $K_2$ bound disks that do not intersect the graph in their interiors.  
No pair of edges between the two subgraphs $K_1$ and $K_2$ connects two of the linked cycles.  
So there is no 3-link in $f(M)$.   
\end{proof}

Notice that there are many graphs that can be constructed with the double star construction that are a $\Y\nabla$ move away from a graph that is not intrinsically 3-linked.  
Consider $(PG)^*_*G$ for any $G\in\mathcal{PF}$, a similar $\Y\nabla$ move to that is the proof above will produce a graph that is not intrinsically 3-linked.  
More generally, this can be done with any double star construction where the vertex of $A_i$ that is not connected to the vertex $v_j$ is trivalent.


\newpage
\begin{thebibliography}{mmI3L}

\bibitem{BF}  G. Bowlin and J. Foisy, {\it Some new intrinsically 3-linked
graphs}, J. of Knot Theory Ramifications {\bf13}(8) (2004),
1021--1027.

\bibitem{BDLST}  A. Brouwer, R. Davis, A. Larkin, D. Studenmund, C. Tucker, {\it Intrinsically $S^1$ 3-linked
graphs and other aspects of $S^1$ embeddings}, Rose-Hulman Undergrad. Math. J. {\bf 8} (2007).


\bibitem{CG} J. Conway and C. Gordan, {\it
Knots and links in spatial graphs}, J. of Graph Theory {\bf 7}
(1983), 445--453.

\bibitem{FFNP} E. Flapan, J. Foisy, R. Naimi, and J. Pommersheim,
{\it Intrinsically n-linked graphs}, J. of Knot Theory
Ramifications {\bf 10}(8) (2001), 1143--1154.

\bibitem{FN} E. Flapan, and R. Naimi, {\it The Y-triangle move does not preserve intrinsic knottedness}, Osaka J. Math. {\bf 45} (2008) 107-111.  

\bibitem{FNP} E. Flapan, R. Naimi, and J. Pommersheim, {\it
Intrinsically triple linked complete graphs}, Topol. Appl. {\bf
115}
 (2001), 239--246.

\bibitem{MRS} R. Motwani, A. Raghunathan and H. Saran, {\it
Constructive results for graph minors:  Linkless embeddings},  29th Annual 
Symposium on Foundations of Computer Science, IEEE (1988), 398--409.


\bibitem{RST} N. Robertson, P. Seymour, and R. Thomas, {\it Sachs'
 linkless embedding conjecture}, J. of Combinatorial Theory, Series B
 {\bf64} (1995), 185--227.

\bibitem{Sac} H. Sachs, {\it On a spatial analogue of Kuratowski's
Theorem on planar graphs -- an open problem}, Graph Theory,
Lag$\acute{o}$w, 1981, Lecture Notes in Mathematics, Vol. 1018
(Springer-Verlag, Berlin, Heidelberg, 1983), 649--662.

\bibitem{Sa} H. Sachs, {\it On spatial representations of finite graphs},
 Colloq. Math. Soc. J\'{a}nos Bolyai, Vol. 37 (North-Holland, Budapest, 1984), 649--662.

\end{thebibliography}
\end{document}